\documentclass{amsart}
\usepackage{graphicx,amsmath}
\usepackage{amssymb,amsthm}

\newtheorem{thm}{Theorem}[section]

\newtheorem{lem}[thm]{Lemma}
\newtheorem{prop}[thm]{Proposition}
\theoremstyle{definition}
\newtheorem{defn}[thm]{Definition}
\theoremstyle{remark}
\newtheorem{rem}[thm]{Remark}
\numberwithin{equation}{section}

\newcommand{\norm}[1]{\left\Vert#1\right\Vert}

\newcommand{\be}{\begin{equation*}}
\newcommand{\ee}{\end{equation*}}

\begin{document}
\author{Beyaz Ba\c{s}ak Koca}
\address[Beyaz Ba\c{s}ak Koca]{Department of Mathematics, Faculty of Science \.{I}stanbul University, Veznec\.{I}ler, 34134, \.{I}stanbul, Turkey }
\email{basakoca@istanbul.edu.tr}

\author{S\.{I}bel \c{S}ah\.{I}n}
\address[Sibel \c{S}ahin]{Department of Natural and Mathematical Sciences,\"{O}zye\v{g}\.{I}n University, \c{C}ekmek\"{o}y, 34794, \.{I}stanbul, Turkey }
\email{sahinsibel@sabanciuniv.edu}

\title[Beurling-type invariant subspaces]{Beurling-type invariant subspaces of the Poletsky-Stessin Hardy spaces in the bidisc}

\dedicatory{Dedicated to Prof.Dr. Naz{\i}m Sad{\i}k on the occasion
of his 65th birthday}

 \maketitle
\date{\today}

\begin{abstract}
The invariant subspaces of the Hardy space on $H^2(\mathbb{D})$ of
the unit disc are very well known however in several variables the
structure of the invariant subspaces of the classical Hardy spaces
is not yet fully understood. In this study we examine the invariant
subspace problem for Poletsky-Stessin Hardy spaces which is a
natural generalization of the classical Hardy spaces to hyperconvex
domains in $\mathbb{C}^n$. We showed that not all invariant
subspaces of $H^{2}_{\tilde{u}}(\mathbb{D}^2)$ are of Beurling-type.
To characterize the Beurling-type invariant subspaces of this space
we first generalized the Lax-Halmos theorem of vector valued Hardy
spaces to the vector valued Poletsky-Stessin Hardy spaces and then
we give a necessary and sufficient condition for the invariant
subspaces of $H^{2}_{\tilde{u}}(\mathbb{D}^2)$ to be of
Beurling-type.
\end{abstract}

\section*{Introduction}

In \cite{Beurling}, Beurling described all invariant subspaces for
the operator of multiplication by $z$ on the Hardy-Hilbert space
$H^2(\mathbb{D})$ of the unit disk. In $H^2(\mathbb{D})$, all
invariant subspaces are of Beurling-type i.e. they are of the form
$fH^2(\mathbb{D})$ where $f$ is an inner function in
$H^2(\mathbb{D})$. However, in several variable case the structure
of the invariant subspaces cannot be characterized in such a simple
form. Although it is quite clear that the Beurling-type subspaces,
are invariant; it is known that not all invariant subspaces are of
this form. In \cite{Jacewicz}, Jacewicz gave an example of an
invariant subspace which can be generated by two functions but
cannot be generated by a single function. Later, Rudin \cite{Rudin}
gave an example of an invariant subspace which cannot be generated
by finitely many elements. There are several studies that try to
characterize the Beurling-type invariant subspaces of
$H^2(\mathbb{D}^{2})$ and in this study we are going to generalize
one of these studies given by Sad{\i}kov in \cite{Sadik}.\\ In 2008,
Poletsky and Stessin introduced Poletsky-Stessin Hardy spaces and
generalized the theory of Hardy spaces to hyperconvex domains in
$\mathbb{C}^n$. The structure of these spaces is examined in detail
in \cite{Poletsky,Sahin,Alan,Shresta}. Hence, it is natural to ask
the invariant subspace problem in the case of Poletsky-Stessin Hardy
spaces. In the case of unit disk, Alan and G\"{o}\v{g}\"{u}\c{s}
\cite{Alan} showed that all invariant subspaces of the
Poletsky-Stessin Hardy space $H^2_{u}(\mathbb{D})$ are of
Beurling-type. In this study we are going to consider the
multivariable case for the Poletsky-Stessin Hardy space
$H^{2}_{\tilde{u}}(\mathbb{D}^2)$ of the bidisc. First of all using
analogous methods to Jacewicz we will show that there exists an
invariant subspace of $H^{2}_{\tilde{u}}(\mathbb{D}^2)$ which is not
of Beurling-type. Then, we are going to generalize the classical
Lax-Halmos theorem to $H^{2}_{\tilde{u}}(\mathbb{D}^2)$ using the
methods of vector valued Hardy spaces. Lastly, we are going to
characterize the Beurling-type invariant subspaces of
$H^{2}_{\tilde{u}}(\mathbb{D}^2)$ by generalizing the ideas of
Sad{\i}kov \cite{Sadik} that he used for the Beurling-type invariant
subspaces of classical Hardy space $H^{2}(\mathbb{D}^2)$.

\section{Preliminaries}
In this section we will give the preliminary definitions and some
important results that we will use throughout this study. Before
proceeding with Poletsky-Stessin Hardy spaces let us first recall
the classical Hardy spaces of the polydisc given in \cite{Rudin}:
\begin{defn}
Hardy spaces on the unit polydisc of $\mathbb{C}^{n}$ are defined
for $1\leq p\leq\infty$ as :
\begin{equation*}
H^{p}(\mathbb{D}^{n})=\{f\in\mathcal{O}(\mathbb{D}^{n}):\sup_{0<r<1}(\frac{1}{(2\pi)^{n}}\int_{\mathbb{T}^{n}}|f(rz)|^{p}d\mu)^{\frac{1}{p}}<\infty\}
\end{equation*}
where $\mathbb{T}^{n}$ is torus and $\mu$ is the usual product
measure on the torus. And
\begin{equation*}
H^{\infty}(\mathbb{D}^{n})=\{f\in\mathcal{O}(\mathbb{D}^{n}):\sup_{z\in\mathbb{D}^{n}}|f(z)|<\infty\}
\end{equation*}
\end{defn}
In \cite{Poletsky}, Poletsky \& Stessin introduced new Hardy type
classes of holomorphic functions on hyperconvex domains $\Omega$ in
$\mathbb{C}^{n}$. Before defining these new classes let us first
give some preliminary definitions. Let $\varphi:\Omega\rightarrow
[-\infty,0)$ be a negative, continuous, plurisubharmonic exhaustion
function for $\Omega$. Following \cite{Demailly} we define the
pseudoball:
\begin{equation*}\label{eq:pseudoball}
B(r)=\{z\in\Omega:\varphi(z)<r\}\quad ,\quad r\in[-\infty,0),
\end{equation*}
and pseudosphere:
\begin{equation*}\label{eq:pseudosphere}
S(r)=\{z\in\Omega:\varphi(z)=r\}\quad ,\quad  r\in[-\infty,0),
\end{equation*}
and set
\begin{equation*}
\varphi_{r}(z)= \max\{\varphi(z),r\}\quad ,\quad r\in(-\infty,0).
\end{equation*} \\ In \cite{Demailly}, Demailly introduced the Monge-Amp\`ere
measures in the sense of currents as :
\begin{equation*}\label{eq:mameasure}
\mu_{\varphi,r}=(dd^{c}\varphi_{r})^{n}-\chi_{\Omega\setminus
B(r)}(dd^{c}\varphi)^{n}\quad r\in(-\infty,0).
\end{equation*}
In this study we will use the boundary value characterization of
Poletsky-Stessin Hardy spaces in most of the results so let us also
mention boundary measures which were introduced by Demailly in
\cite{Demailly2}. Now let $\varphi:\Omega\rightarrow[-\infty,0)$ be
a continuous, plurisubharmonic exhaustion for $\Omega$ and suppose
that the total Monge-Amp\`ere mass is finite that is, we assume that
\begin{equation}
MA(\varphi)=\int_{\Omega}(dd^{c}\varphi)^{n}<\infty.
\end{equation}
Then as $r$ approaches to 0, $\mu_{\varphi,r}$ converges to a
positive measure $\mu_{\varphi}$ weak*-ly on $\Omega$ with total
mass $\int_{\Omega}(dd^{c}\varphi)^{n}$ and supported on
$\partial\Omega$. This measure $\mu_{\varphi}$ is called the
\textbf{Monge-Amp\`ere measure on the boundary associated with the
exhaustion $\varphi$}. As a consequence of (\cite{Sahin},
Proposition 2.2.3) we know that the boundary Monge-Amp\`ere measure
$d\mu_{u}$ is mutually absolutely continuous with respect to the
Euclidean measure on the unit circle and we have,
\begin{equation}\label{rad-nyk der}
d\mu_{u}=\beta(\theta)d\theta
\end{equation}
for a positive $L^1$ function $\beta$ which is defined as

 \[\beta(\theta)=\int_{\mathbb{D}}P(z,e^{i\theta})dd^cu(z).\]
Now we can introduce the Poletsky-Stessin Hardy classes, which will
be our main focus throughout this study. In \cite{Poletsky},
Poletsky \& Stessin gave the definition of new Hardy type spaces
using Monge-Amp\'ere measures as :
\begin{defn}
$H_{\varphi}^{p}(\Omega)$ for $p>0$, is the space of functions
$f\in\mathcal{O}(\Omega)$ such that
\begin{equation*}
\limsup_{r\to 0^{-}}\int_{S_{\varphi,}(r)}|f|^{p}
d\mu_{\varphi,r}<\infty.
\end{equation*}
\end{defn}
The norm on these spaces is given by:
\begin{equation*}
\|f\|_{H_{\varphi}^{p}}=\left(\lim_{r\to
0^{-}}\int_{S_{\varphi}(r)}|f|^{p}
d\mu_{\varphi,r}\right)^{\frac{1}{p}}
\end{equation*}and with respect to these norm the spaces
$H_{\varphi}^{p}(\Omega)$ are Banach spaces \cite{Poletsky}.\\
In Poletsky-Stessin Hardy spaces of the unit disk we have the
canonical inner-outer factorization analogous to the classical Hardy
space case (\cite{sah1}, Theorem 4.2) but throughout this study we
will consider a special type of inner functions which is defined in
\cite{Alan} as follows:
\begin{defn}
Let $u$ be a continuous, subharmonic exhaustion function for
$\mathbb{D}$. A function $\phi\in H^2_u(\mathbb{D})$ is a
\emph{$u$-inner} function if $|\phi^*(\xi)|^{2}\beta(\xi)$ equals to
$1$ for almost every $\xi\in \mathbb{T}$ where $\beta$ is the
function given in (\ref{rad-nyk der}).
\end{defn}

\begin{rem}
The set of $u$-inner functions is non-trivial. In fact first all we
need to show that there is a holomorphic function $\phi$ such that
$|\phi^*|=\frac{1}{\sqrt{\beta}}$. Now since
$\beta(\xi)=\int_{\mathbb{D}}P(z,\xi)dd^cu(z)$ it is a strictly
positive function, $\beta(\xi)>c$ for some $c>0$ so
$\frac{1}{\sqrt{\beta}}$ is a bounded, positive function. Then by
\cite[3.5.1]{Rudin} we know that we have an analytic function
\[\phi(z)=\exp\left\{\int_{\mathbb{T}}\frac{\xi+z}{\xi-z}\log\frac{1}{\sqrt{\beta}}\,d\xi\right\},\;\; z\in \mathbb{D}\]
so that $\phi\in H^\infty(\mathbb{D})$ and
$|\phi^*|=\frac{1}{\sqrt{\beta}}$ a.e. on $\mathbb{T}$. Then
$|\phi^*|^2\beta=1$ a.e. on $\mathbb{T}$ and
$H^\infty(\mathbb{D})\subset H_u^2(\mathbb{D})$ so $\phi\in
H_u^2(\mathbb{D})$.
\end{rem}

In the following sections our main focus will be on the Poletsky-Stessin Hardy space, $H^{2}_{\tilde{u}}(\mathbb{D}^{2})$ of the bidisc generated by the following special type of exhaustion function:\\
Let $u$ be an exhaustion function of the unit disc $\mathbb{D}$ with
finite Monge-Amp\`ere mass. Then the following plurisubharmonic
function,
\begin{equation*}
\tilde{u}(z,w)=\max\{u(z),u(w)\}
\end{equation*}
is an exhaustion for the unit bidisc $\mathbb{D}^{2}$. For this
exhaustion function $\tilde{u}$, the corresponding boundary
Monge-Amp\`ere measure on the torus $\mathbb{T}^2$ is given as
follows (\cite{Sahin}, Theorem 3.2.1) : \be
d\mu_{\tilde{u}}(\theta_{1},\theta_{2})=d\mu_{u}(\theta_{1})d\mu_{u}(\theta_{2})=\beta(\theta_{1})\beta(\theta_{2})d\theta_{1}d\theta_{2}.
\ee

By \cite[3.5.2]{Rudin} we can generalize the definition of $u$-inner
function to $\mathbb{D}^2$ so for a plurisubharmonic exhaustion
function $\tilde u$ of $\mathbb{D}^2$ we can find a holomorphic
function $\phi$ on $\mathbb{D}^2$ (which is called $\tilde u$-inner)
such that $\phi\in H_{\tilde u}^2(\mathbb{D}^2)$ and
$|\phi^*(\xi,\eta)|^2\beta(\xi)\beta(\eta)=1$ a.e. on $\mathbb{T}^2$
where $d\mu_{\tilde{u}}(\xi,\eta)=\beta(\xi)\beta(\eta)d\xi d\eta$.

We call $M$ an invariant subspace of
$H^{2}_{\tilde{u}}(\mathbb{D}^{2})$ if (a) $M$ is a closed linear
subspace of $H^{2}_{\tilde{u}}(\mathbb{D}^{2})$ and (b) $f\in M$
implies $zf\in M$ and $wf\in M$, i.e., multiplication by polynomials
maps $M$ into $M$. An invariant subspace $M$ is called Beurling type
if it is of the form $M=\phi H^2(\mathbb{D}^2)$, where $\phi$ is
$\tilde u$-inner.

In one variable case Alan and G\"{o}\v{g}\"{u}\c{s} \cite[Theorem
3.2]{Alan} extended the classical characterization of invariant
subspaces to the Poletsky-Stessin Hardy spaces as follows:

\begin{thm}
Let $M\neq \{0\}$ be an invariant subspace of
$H^{2}_{u}(\mathbb{D})$. Then there exists a \emph{$u$-inner}
function $\phi$ so that $M=\phi H^2(\mathbb{D})$.
\end{thm}

\section{Main Results}
In this section we first show that the Poletsky-Stessin Hardy space
on the bidisc has an invariant subspaces which is not of the form
$fH^2(\mathbb{D}^2)$ for any $f\in H_{\tilde u}^2(\mathbb{D}^2)$ in
contrast to one variable case. Before proceeding, we recall that
$H^2(\mathbb{D}^2)$ can be seen as a closed subspace
$H^2(\mathbb{T}^2)$ of the standard Lebesgue space
$L^2(\mathbb{T}^2)$ which consists of the functions in
$L^2(\mathbb{T}^2)$ with Fourier coefficients vanishing off a pair
of nonnegative integers. To each function $f$ in $H^2(\mathbb{T}^2)$
with Fourier series $\sum_{m,n=0}^\infty
a_{mn}e^{im\theta_1}e^{in\theta_2}$ we associate the function
$\sum_{m,n=0}^\infty a_{mn}z^mw^n$ analytic on $\mathbb{D}^2$ which
we also denote by $f$. For more details, see \cite{Rudin}. Note that
since $H_{\tilde u}^2(\mathbb{D}^2)$ is a subspace of
$H^2(\mathbb{D}^2)$ by \cite[p.54]{Sahin}, every function in
$H_{\tilde u}^2(\mathbb{D}^2)$ also has the Fourier representation
above.

\begin{thm}\label{example}
There exists an invariant subspace $M$ of $H_{\tilde
u}^2(\mathbb{D}^2)$ which is of the form
$M=f_1H^2(\mathbb{D}^2)+f_2H^2(\mathbb{D}^2)$ for some $f_1,f_2\in
H_{\tilde u}^2(\mathbb{D}^2)$ but can not be of the form
$M=hH^2(\mathbb{D}^2)$ for any $h\in H_{\tilde u}^2(\mathbb{D}^2)$.
\end{thm}

\begin{proof}

We choose $f_1(z,w)=\phi(z)\phi(w)q(z)$ and
$f_2(z,w)=\phi(z)\phi(w)w$, where $\phi$ is a non-vanishing
$u$-inner function in $H_u^2(\mathbb{D})$, $q$ is a nonconstant,
singular inner function in $H^2(\mathbb{D})$, which means that $q$
never vanishes in $\mathbb{D}$ and has modulus one a.e. on
$\mathbb{T}$. It is clear that $f_2\in H_{\tilde u}^2(\mathbb{D}^2)$
and since

\begin{equation*}
\begin{split}
||f_1||^2_{H_{\tilde u}^2(\mathbb{D}^2)} & = \int_{\mathbb{T}}\int_{\mathbb{T}}|\phi^*(z)\phi^*(w)q^*(z)|^2d\mu_u(z)d\mu_u(w)\\
 & =
 \int_{\mathbb{T}}\int_{\mathbb{T}}|\phi^*(z)|^2|\phi^*(w)|^2|q^*(z)|^2d\mu_u(z)d\mu_u(w)\\
 &=\int_{\mathbb{T}}\int_{\mathbb{T}}|q^*(z)|^2dzdw\\
 &= \int_{\mathbb{T}}\int_{\mathbb{T}}dzdw  < \infty,
\end{split}
\end{equation*}
$f_1\in H_{\tilde u}^2(\mathbb{D}^2)$. Consider
$M=f_1H^2(\mathbb{D}^2)+f_2H^2(\mathbb{D}^2)$. It is easily seen
that $M$ is an invariant subspace of $ H_{\tilde
u}^2(\mathbb{D}^2)$.

Suppose that $M$ is of the form $M=hH^2(\mathbb{D}^2)$ for any $h\in
H_{\tilde u}^2(\mathbb{D}^2)$. Let $H^2(S_1)$ denote the subspace of
$L^2(\mathbb{T}^2)$ consisting of functions whose Fourier
coefficients vanish off the half-plane
$S_1=\{(m,n)\in\mathbb{Z}^2:m>0\}\cup
\{(0,n)\in\mathbb{Z}^2:n\geq0\}$. It is clear that $M_1:=hH^2(S_1)$
is the invariant subspace of $H^2(S_1)$. If $q$ and $\phi$ have the
form $q(z)=\sum_{m=0}^\infty a_mz^m$ and $\phi(z)=\sum_{m=0}^\infty
b_mz^m$ respectively, then we see that
\[b_0^2a_0=\phi(z)\phi(w)\left(\sum_{m=0}^\infty a_mz^m-\sum_{m=1}^\infty a_m z^m w^{-1}w\right)=f_1-\sum_{m=1}^\infty a_m(z^m w^{-1})f_2\]
lies in $M_1$. Because $c_m=(m,-1)\in S_1$ for $m\geq 1$ and so
$c_mf_2=z^mw^{-1}f_2\in M_1$. Since $q$ is a singular function and
$\phi$ is non-vanishing, $b_0^2a_0=\phi(0)\phi(0)q(0)\neq 0$, so
that the constant functions lie in $M_1$. Thus
$M_1=hH^2(S_1)=H^2(S_1)$. This property of $h$ for any half-plane
containing the support of Fourier transform of $h$ is equivalent to
an analytic condition independent of the half-plane
\cite[p.128]{Jacewicz}. In particular $hH^2(S_2)=H^2(S_2)$ for the
half-plane $S_2=\{(m,n)\in\mathbb{Z}^2:n>0\}
\cup\{(m,0)\in\mathbb{Z}^2: m\geq0\}$.

Let $P$ be the orthogonal projection of $H^2(S_2)$ onto
$H_u^2(\mathbb{D})$ (Remark that the Fourier coefficients of the
element of $H_u^2(\mathbb{D})$ are zero for $m<0$). The invariant
subspaces of the form $f_1H^2(\mathbb{D}^2)+f_2H^2(\mathbb{D}^2)$
and $hH^2(\mathbb{D}^2)$ are the same. Since $S_2$ contains the set
$\{(m,n): m\geq 0,n\geq 0\}$, the invariant subspaces
$f_1H^2(S_2)+f_2H^2(S_2)$ and $hH^2(S_2)$ are the same. These
subspaces are denoted by $M_2(f_1,f_2)$ and $M_2(h)$, respectively.
$P[M_2(f_1,f_2)]$ is the closed linear span of all $z^m\phi(z)q(z)$,
for $m\geq 0$, while $P[M_2(h)]=H_u^2(\mathbb{D})$. Thus by the
definition of $f_1$ from $q$, it is obtained
$\phi(z)q(z)H^2(\mathbb{D})=H_u^2(\mathbb{D})$. In view of the
equality $H_u^2(\mathbb{D})=\phi(z)H^2(\mathbb{D})$, we have
$qH^2(\mathbb{D})=H^2(\mathbb{D})$, i.e., $q$ is outer in
$H^2(\mathbb{D})$. This is contradiction and so $M$ can not be of
the form $M=hH^2(\mathbb{D}^2)$ for any $h\in H_{\tilde
u}^2(\mathbb{D}^2)$.
\end{proof}

As a consequence of this theorem, we have that not all invariant
subspaces of $H_{\tilde u}^2(\mathbb{D}^2)$ are Beurling-type. Then
it is natural to ask the structure of Beurling type invariant
subspaces of $H_{\tilde u}^2(\mathbb{D}^2)$.\\

First of all, we need to recall the class of vector-valued analytic
functions. Let $K$ be a Hilbert space with inner product
$\langle\cdot,\cdot\rangle$ and norm $||\cdot||_K$. Then by $H^2(K)$
we mean the space of all $K$-valued holomorphic functions
$f(z)=\sum_{n=0}^\infty a_nz^n$ on $\mathbb{D}$ for which the
quantity

 \be
\displaystyle{\frac{1}{2\pi}\int_{0}^{2\pi}\|f(re^{it})\|^2_{K}dt}=\sum_{n=0}^\infty
\|a_n\|^2_K r^{2n} \ee remains bounded for $0\leq r<1$. Clearly,
$H^2(K)$ is a Hilbert space under the inner product
\[\langle f, g\rangle_2=\lim_{r\rightarrow 1}\displaystyle{\frac{1}{2\pi}}\int_{0}^{2\pi}\langle f(re^{i\theta}),g(re^{i\theta})\rangle_K d\theta=\sum_{n=0}^\infty\langle a_n,b_n\rangle_K\]
for any $f(z)=\sum_{n=0}^\infty a_nz^n$ and $ g(z)=\sum_{n=0}^\infty
b_nz^n$ in the space. Now if $K$ is a reflexive Banach space then it
has Fatou Property i.e. each $f\in H^{1}(K)$ has non-tangential
limits on $\partial K$ (\cite[pg:38, 48]{Aytuna}). Hence, we know
that each $f\in H^2(K)$ has the radial limit $f^{*}$ as a Bochner
measurable function and $f^{*}\in L^2_+(K)$, where $L^2_+(K)$ is the
space of $L^2(K)$ functions whose negative Fourier coefficients are
0, and we also have $\|f\|_{H^2(K)}=\|f^{*}\|_{L^2_+(K)}$ (For
details see \cite[pg:183-186]{Nagy}).

On the other hand, if $B(K,K_1)$ denotes the algebra of all the
bounded linear operators from $K$ to $K_1$, then by
$H^\infty(B(K,K_1))$ we mean the algebra of bounded $B(K,
K_1)$-valued holomorphic functions $\Theta$ on $\mathbb{D}$ in the
norm
$\|\Theta\|_\infty=\sup_{z\in\mathbb{D}}\norm{\Theta(z)}_{B(K,K_1)}<\infty$.
It is obvious that each $\Theta\in H^\infty(B(K,K_1))$ gives rise to
a bounded linear operator from $H^2(K)$ into $H^2(K_1)$ namely, to
an element $\Theta$, we correspond an operator
$\hat{\Theta}:H^2(K)\rightarrow H^2(K_1)$ that is defined by the
formula
\[(\hat{\Theta}f)(z)=\Theta(z)f(z),\, z\in \mathbb{D},\, f\in H^2(K).\]

An operator-valued  $\Theta\in H^\infty(B(K,K_1)) $ is called inner
if $\Theta(e^{it})$ is an isometry from $K$ into $K_1$ for almost
every $t$ or equivalently, the operator $\hat{\Theta}$ is an
isometry.

The reader can find the details of vector-valued analytic functions
in \cite{Radjavi, Nagy, Rosenblum}.\\

Analogously, we are going to define the vector valued
Poletsky-Stessin Hardy spaces as follows:

\begin{defn}
Let $K$ be a Hilbert space, $u$ be a continuous, subharmonic
exhaustion function for $\mathbb{D}$. Then the vector valued
Poletsky-Stessin Hardy space is defined as follows: \be
H^{2}_{u}(K)=\{f:\mathbb{D}\rightarrow K, \textit{holomorphic}:
\sup_{r<0}\displaystyle{\int_{S_{u}(r)}\|f(z)\|^{2}_{K}d\mu_{u,r}(z)}<\infty
\} \ee
\end{defn}
Following step by step the same arguments from the scalar valued
case one can easily see that $H^{2}_{u}(K)\subset H^2(K)$. Thus, we
automatically inherit the radial boundary values from the classical
Hardy space $H^2(K)$ and again just rewriting scalar value arguments
we have the following boundary value characterization:
\begin{prop}
Let $f\in H^{2}_{u}(K)$ and $f^*$ be its radial boundary value
function. Then \be
\|f\|^{2}_{H^{2}_{u}(K)}=\|f^{*}\|^{2}_{L^2_{+,u}(K)}=\int_{\mathbb{T}}\|f^{*}(\xi)\|^2_{K}d\mu_{u}(\xi)
\ee
\end{prop}
\begin{proof}
Directly follows from the scalar valued argument given in
(\cite{Sahin}, Theorem 2.2.1).
\end{proof}

Now, recall the Wold decomposition for isometries \cite[p.3, Theorem
1.1]{Nagy}: Let $V$ be an arbitrary isometry on a Hilbert space $H$.
Then $H$ decomposes into an orthogonal sum $H=H_1\oplus H_2$ such
that $H_1$ and $H_2$ reduce $V$, the part of $V$ on $H_1$ is unitary
and the part of $V$ on $H_2$ is a unilateral shift. This
decomposition is uniquely determined, indeed we have
\[H_1=\bigcap_{n=0}^\infty V^n H \mbox{ and } H_2=\bigoplus_{n=0}^\infty V^n E \mbox{ where } E=H\ominus VH.\]
The space $H_1$ or $H_2$ may be absent, i.e., equal to $\{0\}$.\\

If the Poletsky-Stessin Hardy space over the  bidisc is interpreted
as the vector-valued analytic functions on the unit disc of complex
plane, then invariant subspaces under the multiplication
operator by the independent variable $z$ are described in terms of Lax-Halmos theorem.\\

\begin{thm}\label{lax-halmos}
Let $M$ be a non-zero subspace of $H_u^2(H_u^2(\mathbb{D}))$.
$M$ is invariant under the multiplication operator by the
independent variable $z$ if and only if there exists a Hilbert space $E$ and
an inner function $\Theta\in H^\infty(B(E, \varphi
H^2_u(\mathbb{D})))$ such that $M=\hat\Theta H^2(E)$. This class of
the functions $\Theta$ is denoted by $\{\Theta_M\}$.
\end{thm}

First of all, we need the following lemma:
\begin{lem}\label{*}
$H_u^2(H_u^2(\mathbb{D}))=\varphi H^2(H_u^2(\mathbb{D}))$ where
$\varphi$ is the $u$-inner function which gives
$H_u^2(\mathbb{D})=\varphi H^2(\mathbb{D})$.
\end{lem}

\begin{proof}
Let $f\in\varphi H^2(H_u^2(\mathbb{D}))$. Then $f(z)=\varphi(z)h(z)$
where $h(z)\in H_u^2(\mathbb{D})$. Now
\begin{equation*}
\begin{split}
\int_{\mathbb{T}}||f(z)||^2_{H_u^2(\mathbb{D})}d\mu_u(z)&=\int_{\mathbb{T}}|\varphi(z)|^2||h(z)||^2_{H_u^2(\mathbb{D})}d\mu_u(z)\\
 &=\int_{\mathbb{T}}||h(z)||^2_{H_u^2(\mathbb{D})}d\theta=||h||_{H^2(H_u^2(\mathbb{D}))}<\infty.
\end{split}
\end{equation*}

Then $f\in H^2(H_u^2(\mathbb{D}))$ and
$H_u^2(H_u^2(\mathbb{D}))\supseteq\varphi H^2(H_u^2(\mathbb{D}))$.
Conversely, let $f\in H_u^2(H_u^2(\mathbb{D}))$. Then consider the
function $\displaystyle{\frac{f(z)}{\varphi(z)}}$. We want to show
that $\displaystyle{\frac{f(z)}{\varphi(z)}=h(z)}$ is in
$H^2(H_u^2(\mathbb{D}))$. First of all for all $z\in\mathbb{D}$
$h(z)=\displaystyle{\frac{f(z)}{\varphi(z)}=\frac{h_z}{\varphi(z)}}\in
H_u^2(\mathbb{D})$ since $h_z\in H_u^2(\mathbb{D})$ and
\begin{equation*}
\begin{split}
\int_{\mathbb{T}}||h(z)||^2_{H_u^2(\mathbb{D})}d\theta & = \int_{\mathbb{T}}\frac{1}{|\varphi(z)|^2}\int_{\mathbb{D}}|h_z(w)|^2d\mu_u(w)d\theta\\
 &=\int_{\mathbb{D}}\int_{\mathbb{D}}|h_z(w)|^2d\mu_u(w)d\mu_u(z)\\
 &=\int_{\mathbb{T}}||f(z)||^2_{H_u^2(\mathbb{D})}d\mu_u(z)=||f||<\infty.
\end{split}
\end{equation*}
since $|\varphi|^2\beta=1$ a.e. we have $|\varphi|^2d\mu_u=d\theta$.
Hence we obtain $H_u^2(H_u^2(\mathbb{D}))\subseteq\varphi
H^2(H_u^2(\mathbb{D}))$.
\end{proof}

\begin{proof}[Proof of Theorem \eqref{lax-halmos}]
If $\Theta\in H^\infty(B(E, \varphi
H^2_u(\mathbb{D})))$ is an inner function then the corresponding operator is isometric and hence $M=\hat\Theta H^2(E)$ is closed. Its invariance for the multiplication by $z$ is obvious.\\
Now let $M$ be an invariant subspace of $H_u^2(H_u^2(\mathbb{D}))$ under
multiplication by $z$. Now first of all embedding
$H_u^2(\mathbb{D})$ in $H^2(H_u^2(\mathbb{D}))$ as a subspace by
identifying the element $\lambda\in H_u^2(\mathbb{D})$ with the
constant function $\lambda(z)=\lambda$; $H^2_u(\mathbb{D})$ is then
wandering for the multiplication operator by $z$ and
\[H^2(H^2_u(\mathbb{D}))=\bigoplus_{n=0}^\infty z^nH^2_u(\mathbb{D})\]
and by Lemma \eqref{*} we have
\[H_u^2(H^2_u(\mathbb{D}))=\varphi H^2(H_u^2(\mathbb{D}))=\bigoplus_{n=0}^\infty z^n(\varphi H^2_u(\mathbb{D})).\]
Let $V$ denote the restriction of the multiplication operator by $z$
to the invariant subspace $M$; this is an isometry on $M$. We have
\[\bigcap_{n=0}^\infty V^nM\subset \bigcap_{n=0}^\infty z^nH_u^2(H^2_u(\mathbb{D}))\subset \bigcap_{n=0}^\infty z^nH^2(H^2_u(\mathbb{D}))=\{0\}\]
and thus $V$ has no unitary part so that the corresponding
Wold-decomposition of the form $M=\bigoplus_{n=0}^\infty V^n E$,
where $E=M\ominus (VM)$. Let us now apply Lemma 3.2
\cite[p.195]{Nagy} to $R_+=M$, $U_+=V$, $U=E$,
$R_+'=H_u^2(H_u^2(\mathbb{D}))$, $U_+=\mbox{multiplication by $z$}$,
$U'=\varphi H^2_u(\mathbb{D})$ and $Q=\mbox{the identical
transformation of $M$ into $H_u^2(H_u^2(\mathbb{D}))$ }$, then there
exists an inner function $\Theta\in H^\infty(B(E, \varphi
H^2_u(\mathbb{D})))$ such that
\begin{equation}\label{**}
\phi_+^{\varphi H_u^2(\mathbb{D})}Q=\hat{\Theta}\phi_+^E
\end{equation}
on $M$. Since $\varphi H^2_u (\mathbb{D})$ consists of the constant
functions in $H_u^2(H_u^2(\mathbb{D}))$, the Fourier representation
of $H_u^2(H_u^2(\mathbb{D}))$  with respect to multiplication by $z$
is identity transformation. On the other hand we have $Qh=h$ for
$h\in M$. Thus \eqref{**} reduces to $h=\hat{\Theta}\phi_+^Eh$,
$h\in M$ and hence we have $M=\hat{\Theta}\phi_+^E M= \hat{\Theta}
H^2(E)$ as claimed.
\end{proof}

\begin{lem}\label{vector-valued}
Vector valued Poletsky-Stessin Hardy space
$H_u^2(H_u^2(\mathbb{D}))$ is isometrically isomorphic to the
Poletsky-Stessin Hardy space $H_{\tilde{u}}^2(\mathbb{D}^2)$ of
bidisc.
\end{lem}
\begin{proof}
Let $\tilde{u}(z,w)=\max\{u(z),u(w)\}$ be the exhaustion function for the bidisc $\mathbb{D}^{2}$ then we have the following isometric isomorphism between the Banach spaces $H_u^2(H_u^2(\mathbb{D}))$ and $H_{\tilde{u}}^2(\mathbb{D}^2)$:\\
Take $g\in H_u^2(H_u^2(\mathbb{D}))$ then $g(z)=g_{z}(w)$ for some
$g_{z}\in H^{2}_{u}(\mathbb{D})$. Now consider the corresponding
function $\bar{g}$ on $\mathbb{D}^{2}$ defined as
$\bar{g}(z,w)=g_{z}(w)$ then using (\cite{Sahin}, Theorem 3.2.1) we
have,
\begin{equation*}
\begin{split}
\|\bar{g}\|^{2}_{H_{\tilde{u}}^2(\mathbb{D}^2)}&=\int_{\mathbb{T}^{2}}|\bar{g}^{*}(\xi,\eta)|^{2}d\mu_{\tilde{u}}(\xi,\eta)\\
 &=\int_{\mathbb{T}}\int_{\mathbb{T}}|\bar{g}^{*}(\xi,\eta)|^{2}d\mu_{u}(\eta)d\mu_{u}(\xi)\\
 &=\int_{\mathbb{T}}\|g_{z}\|_{H^{2}_{u}(\mathbb{D})}^{2}d\mu_{u}=\|g\|^{2}_{H_u^2(H_u^2(\mathbb{D}))}.
\end{split}
\end{equation*}

\end{proof}

Suppose that a subspace $M$ of $H_{\tilde u}^2(\mathbb{D}^2)$ which
is invariant under the multiplication operators by independent
variables $z$ and $w$ is of Beurling-type, i.e., $M$ is of the form
$M=\phi H^2(\mathbb{D}^2)$ for some $u$-inner function $\phi$. Since
$M$ is invariant under the multiplication by $z$, in view of Lemma
\eqref{vector-valued} and Theorem \eqref{lax-halmos}, it can be
described by the class of functions $\{\Theta_M\}$ . However, the
subspaces determined by these class of functions  $\{\Theta_M\}$ are
not generally of Beurling-type and the following theorem gives a
condition for those subspaces which are defined by $\{\Theta_M\}$ to
be of Beurling-type using the simple relation
$H^2(H_u^2(\mathbb{D}))= H^2(N)\oplus H^2(N^\perp)$, where $N$ is a
subspace of $H_u^2(\mathbb{D})$ and $N^\perp$ its complement.

\begin{thm}\label{main}
A subspace $M$ of $H_{\tilde u}^2(\mathbb{D}^2)$ is invariant under
the multiplication operators by the independent variables $z$ and
$w$ is Beurling-type if and only if there exists at least an operator valued
holomorphic function $\Theta(z)$, $z\in\mathbb{D}$ in the class
$\{\Theta_M\}$ such that for every $z_0\in \mathbb{D}$ the operator
$\Theta(z_0)$ on $H_u^2(\mathbb{D})$ commutes with the
multiplication operator by $w$ in $H_u^2(\mathbb{D})$.
\end{thm}
Before starting the proof of the main theorem, we need the following preliminary results (\cite{Shresta}, pp:34):

Define
$$
\tilde{\alpha}(z)=\int_{\mathbb{T}}\frac{e^{i\theta}+z}{e^{i\theta}-z}\log(\beta(e^{i\theta}))d\theta
$$ then the function
$$
A(z)=e^{\tilde{\alpha}(z)}
$$ is a holomorphic function of the unit disc which extends smoothly to the unit circle with the property $|A^{*}(e^{i\theta})|=\beta(e^{i\theta})$.

\begin{thm}[Shresta'15]
The space $H^{p}_{u}(\mathbb{D})$ is isometrically isomorphic to $H^{p}(\mathbb{D})$.
\end{thm}
\begin{rem}
The above mentioned isomorphism is given as follows:
$$
H^{p}_{u}(\mathbb{D})\leftrightarrow H^{p}(\mathbb{D})
$$
$$
f\leftrightarrow A^{\frac{1}{p}}f
$$
\end{rem}

\begin{lem}\label{dilation bidisc}
The set
$\left\{\frac{z^nw^m}{\sqrt{A(z)}\sqrt{A(w)}}\right\}_{n,m\geq0}$
is an orthonormal basis of $H_{\tilde u}(\mathbb{D}^2)$.
\end{lem}
\begin{proof}
It is enough to show that
$\left\{\frac{z^nw^m}{\sqrt{A(z)}\sqrt{A(w)}}\right\}_{n,m\geq0}$
is a complete orthonormal set in $H_{\tilde u}(\mathbb{D}^2)$ by
\cite[Theorem 4.13, p.16]{Conway}. First we will show the orthonormality of this set:
$$
\left\|\frac{z^nw^m}{\sqrt{A(z)}\sqrt{A(w)}}\right\|_{H_{\tilde u}(\mathbb{D}^2)}=\left\langle\frac{z^nw^m}{\sqrt{A(z)}\sqrt{A(w)}},\frac{z^nw^m}{\sqrt{A(z)}\sqrt{A(w)}}\right\rangle
$$
$$
=\int_{\mathbb{T}^{2}}\frac{e^{in\theta_{1}}e^{im\theta_{2}}e^{-in\theta_{1}}e^{-im\theta_{2}}}{|A^{*}(e^{i\theta_{1}})||A^{*}(e^{i\theta_{2}})|}d\mu_{\tilde{u}}
$$
$$
=\left(\int_{\mathbb{T}}\frac{d\mu_{u}(\theta_{1})}{|A^{*}(e^{i\theta_{1}})|}\right)\left(\int_{\mathbb{T}}\frac{d\mu_{u}(\theta_{2})}{|A^{*}(e^{i\theta_{2}})|}\right)=\int_{\mathbb{T}}d\theta_{1}\int_{\mathbb{T}}d\theta_{2}=1
$$
For any $(n_{1},m_{1})$, $(n_{2},m_{2})$ such that $(n_{1},m_{1})\neq(n_{2},m_{2})$,
$$
\left\langle\frac{z^n_{1}w^m_{1}}{\sqrt{A(z)}\sqrt{A(w)}},\frac{z^n_{2}w^m_{2}}{\sqrt{A(z)}\sqrt{A(w)}}\right\rangle
$$
$$
=\left(\int_{\mathbb{T}}\frac{e^{i(n_{1}-n_{2})\theta_{1}}}{|A^{*}(e^{i\theta_{1}})|}d\mu_{u}(\theta_{1})\right)\left(\int_{\mathbb{T}}\frac{e^{i(m_{1}-m_{2})\theta_{2}}}{|A^{*}(e^{i\theta_{2}})|}d\mu_{u}(\theta_{2})\right)
$$
$$
=\left(\int_{\mathbb{T}}e^{i((n_{1}-n_{2})\theta_{1})}d\theta_{1}\right)\left(\int_{\mathbb{T}}e^{i((m_{1}-m_{2})\theta_{2})}d\theta_{2}\right)=0
$$ since $\{z^k\}_{k\geq0}$ is orthonormal in $H^{2}(\mathbb{D})$.

As for completeness let $f(z,w)\in H_{\tilde u}(\mathbb{D}^2)$ be such that
\[\int_{\mathbb{T}^2}f(e^{i\theta_{1}},e^{i\theta_{2}})\frac{e^{-in\theta_{1}}e^{-im\theta_{2}}}{\sqrt{\overline{A^{*}(e^{i\theta_{1}})}}\sqrt{\overline{A^{*}(e^{i\theta_{2}})}}}d\mu_u(\theta_{1})d\mu_u(\theta_{2})=0\]
for all $n,m$. We claim that $f\equiv0$. We have, by Fubini's
theorem,
$$
0=\int_{\mathbb{T}^2}f(e^{i\theta_{1}},e^{i\theta_{2}})\frac{e^{-in\theta_{1}}e^{-im\theta_{2}}}{\sqrt{\overline{A^{*}(e^{i\theta_{1}})}}\sqrt{\overline{A^{*}(e^{i\theta_{2}})}}}d\mu_u(\theta_{1})d\mu_u(\theta_{2})
$$
$$
=\int_{\mathbb{T}}\left(\int_{\mathbb{T}}f(e^{i\theta_{1}},e^{i\theta_{2}})\frac{e^{-in\theta_{1}}}{\sqrt{\overline{A^{*}(e^{i\theta_{1}})}}}d\mu_u(\theta_{1})\right)\frac{e^{-im\theta_{2}}}{\sqrt{\overline{A^{*}(e^{i\theta_{2}})}}}d\mu_u(\theta_{2})
$$
Now, using the fact that $\{\frac{w^m}{\sqrt{A(w)}}\}$ is an
orthonormal basis for $H_u^2(\mathbb{D})$, we have that for all $n$
\[\int_{\mathbb{T}}f(e^{i\theta_{1}},e^{i\theta_{2}})\frac{e^{-in\theta_{1}}}{\sqrt{\overline{A^{*}(e^{i\theta_{1}})}}}d\mu_u(\theta_{1})=0\;\;\mbox{$\mu_u$}-a.e.\]
Let $E_n\subset\mathbb{T}$ be the set of measure zero where the
above equality does not hold and let $E=\bigcup_n E_n$. Then for
$e^{i\theta_{2}}\not\in E$,
\[\int_{\mathbb{T}}f(e^{i\theta_{1}},e^{i\theta_{2}})\frac{e^{-in\theta_{1}}}{\sqrt{\overline{A^{*}(e^{i\theta_{1}})}}}d\mu_u(\theta_{1})=0\]
for all $n$, and thus again using the fact that
$\{\frac{z^n}{\sqrt{A(z)}}\}$ is a complete orthonormal set in
$H_u^2(\mathbb{D})$, we have that $f(e^{i\theta_{1}},e^{i\theta_{2}})=0$ $\mu_u$-a.e. Therefore
$f=0$, $\mu_{\tilde u}$-a.e. which gives $f\equiv0$. Hence the claim follows.
\end{proof}

\begin{lem}\label{commutant}
The set of all bounded linear operators on $H_u^2(\mathbb{D})$
commuting with the operators of multiplication by the independent
variable $z$ is the set of all multiplication operators by
multipliers in $H^\infty(\mathbb{D})$.
\end{lem}

\begin{proof}
The claim is clear since the commutant of the multiplication
operator by independent variable on $H^2(\mathbb{D})$ is the set of
all multiplication operators by multipliers in
$H^\infty(\mathbb{D})$ by \cite[Problem 116]{Halmos} and
$H^2_u(\mathbb{D})$ is subspace of $H^2(\mathbb{D})$.
\end{proof}

\begin{proof}[Proof of Theorem \eqref{main}]
Suppose that there is a $\Theta$ in the class $\{\Theta_M\}$ such
that for any fixed $z_0\in\mathbb{D}$, $\Theta(z_0)$ commutes with
the multiplication operator by $w$ in $H_u^2(\mathbb{D})$. Since, by
Lemma \eqref{commutant}, the commutant of the multiplication
operator by $w$ in $H_u^2(\mathbb{D})$ is $H^\infty(\mathbb{D})$, it
follows that $\Theta(z_0)\in H^\infty(\mathbb{D})$ for every
$z_0\in\mathbb{D}$. Let's note that the function $z\rightarrow
\Theta(z)1$, where the function $1$ in $H_u^2(\mathbb{D})$ is
identically equal to 1, is an analytic function of $z$ taking values
in $H_u^2(\mathbb{D})$. Hence it follows that if $\phi=\Theta 1$,
then $\Theta(z_0)1$ coincides with a function $\phi(z_0,w)$, and the
family of functions $w\rightarrow \Theta(z_0)(w),\;\;w\in\mathbb{D}$
is a family generated by an analytic function $\phi$. To obtain that
$\phi$ is a $\tilde u$-inner function it is enough to show that the
multiplication operator by $\phi$ in $H_{\tilde u}^2(\mathbb{D}^2)$
is an isometry. If $g\in H_{\tilde u}^2(\mathbb{D}^2)$, then for
$g(z,w)=g_z(w)$, by \cite{Sahin}, we have
\[||g||^2=\int_{\mathbb{T}}\int_{\mathbb{T}}|g^*(\xi,\eta)|^2d\mu_u(\xi)d\mu_u(\eta)=\int_{\mathbb{T}}||g_\xi||^2d\mu_u(\eta).\]
Applying this to the function $\phi\varphi$ we obtain

\[||\phi\varphi||^2=\int_{\mathbb{T}}||\phi_\eta\varphi_\eta||^2d\mu_u(\eta)\]
and by assumption $\phi_\eta$ is an isometric operator for almost
all $\eta$, therefore $||\phi_\eta\varphi_\eta||=||\varphi_\eta||$
for almost all $\eta$ and $||\phi\varphi||^2=||\varphi||^2$. Thus
the operator $\Theta$ in $H_u^2(H_u^2(\mathbb{D}))$ and the
multiplication operator by $\phi=\Theta1$ in $H_{\tilde
u}^2(\mathbb{D}^2)$ are bounded operators which agree on vectors the
type $\frac{z^kw^l}{\sqrt{A(z)}\sqrt{A(w)}}$, $k,l\geq 0$ under the canonical isomorphism between
$H_u^2(H_u^2(\mathbb{D}))$ and $H_{\tilde u}^2(\mathbb{D}^2)$. In
Lemma (\ref{dilation bidisc}) we have proved that the elements
$\frac{z^kw^l}{\sqrt{A(z)}\sqrt{A(w)}}$, $k,l\geq0$
are dense in $H_{\tilde u}^2(\mathbb{D}^2)$.
Hence $\phi=\Theta 1$ and $\Theta$ correspond to each other.\\

For the converse direction now suppose that $M$ is a subspace generated by a $\tilde{u}$-inner function $\phi$ then for almost any $\xi \in \mathbb{T}$, $\phi^{*}(\xi, \cdot)$ is a $u$-inner function in $H^{\infty}(\mathbb{D})$ and the radial boundary values of the operator valued function $\Theta(z)$, where $\Theta(z)$ is the operator of multiplication by the function $\phi$, is an isometry almost everywhere. Hence the result follows.
\end{proof}

\end{document}